\documentclass[12pt,a4paper]{amsart}

\usepackage{amssymb}

\usepackage{amscd}
\usepackage{hyperref}

\def\frak{\mathfrak}
\def\Bbb{\mathbb}
\def\Cal{\mathcal}

\let\phi\varphi

\newcommand{\al}{\alpha}

\newcommand{\ga}{\gamma}

\newcommand{\ps}{\psi}

\newcommand{\si}{\sigma}

\newcommand{\Ga}{\Gamma}
\newcommand{\La}{\Lambda}
\newcommand{\Ph}{\Phi}
\newcommand{\Ps}{\Psi}
\newcommand{\Om}{\Omega}

\newcommand{\Ups}{\Upsilon}
\def\Rho{\mbox{\textsf{P}}}

\newcommand{\barm}{\overline{M}}
\newcommand{\vol}{\operatorname{vol}}

\newcommand{\Ric}{\operatorname{Ric}}

\newcommand{\rpl}                         
{\mbox{$
\begin{picture}(12.7,8)(-.5,-1)
\put(0,0.2){$+$}
\put(4.4,3.1){\oval(8,8)[r]}
\end{picture}$}}

\newcounter{theorem}
\newtheorem{thm}[theorem]{Theorem}
\newtheorem*{thm*}{Theorem \thesubsection}
\newtheorem{lemma}[theorem]{Lemma}
\newtheorem{prop}[theorem]{Proposition}
\newtheorem{cor}[theorem]{Corollary}
\newtheorem*{lemma*}{Lemma \thesubsection}
\newtheorem*{prop*}{Proposition \thesubsection}
\newtheorem*{cor*}{Corollary \thesubsection}

\theoremstyle{definition}

\newtheorem*{definition*}{Definition \thesubsection}

\newtheorem*{example*}{Example \thesubsection}
\theoremstyle{remark}

\newtheorem*{remark*}{Remark \thesubsection}

\def\sideremark#1{\ifvmode\leavevmode\fi\vadjust{\vbox to0pt{\vss
 \hbox to 0pt{\hskip\hsize\hskip1em
 \vbox{\hsize3cm\tiny\raggedright\pretolerance10000
  \noindent #1\hfill}\hss}\vbox to8pt{\vfil}\vss}}}%
                        
\begin{document}

\title{Scalar Curvature and Projective Compactness}

\author{Andreas \v Cap and A.\ Rod Gover}

\address{A.\v C.: Faculty of Mathematics\\
University of Vienna\\
Oskar--Morgenstern--Platz 1\\
1090 Wien\\
Austria\\
A.R.G.:Department of Mathematics\\
  The University of Auckland\\
  Private Bag 92019\\
  Auckland 1142\\
  New Zealand} 
\email{Andreas.Cap@univie.ac.at}
\email{r.gover@auckland.ac.nz}

\begin{abstract}
  Consider a manifold with boundary, and such that the interior is
  equipped with a pseudo-Riemannian metric. We prove that, under mild
  asymptotic non-vanishing conditions on the scalar curvature, if the
  Levi-Civita connection of the interior does not extend to the
  boundary (because for example the interior is complete) whereas its
  projective structure does, then the metric is projectively compact
  of order 2; this order is a measure of volume growth toward
  infinity. The result implies a host of results including that the
  metric satisfies asymptotic Einstein conditions, and induces a
  canonical conformal structure on the boundary.  Underpinning this work is
  a new interpretation of scalar curvature in terms of projective
  geometry. This enables us to show that if the projective structure
  of a metric extends to the boundary then its scalar curvature also
  naturally and smoothly extends.
\end{abstract}

\subjclass[2010]{Primary 53A20, 53B10, 53C21; Secondary 35N10, 53A30,58J60}

\thanks{Both authors gratefully acknowledge support from the Royal
  Society of New Zealand via Marsden Grant 13-UOA-018; A\v C
  gratefully acknowledges support by projects P23244--N13 and
  P27072--N25 of the Austrian Science Fund (FWF) and the hospitality
  of the University of Auckland. }

\maketitle

\pagestyle{myheadings} \markboth{\v Cap, Gover}{Scalar Curvature and 
Projective Compactness}

\section{Introduction}

Throughout this article we consider a smooth manifold $\barm$ of
dimension $n+1$ with boundary $\partial M$ and interior $M$, and the
basic topic is that of relating geometric structures on $M$ to geometric
structures (in general of a different type) on $\partial M$. Apart from
their intrinsic interest in differential geometry, questions of this
type play an important role in several other areas of mathematics
(e.g.~scattering theory) and theoretical physics (e.g.~general
relativity and the AdS/CFT--correspondence), see the introduction of
\cite{Proj-comp} for a more detailed discussion. 

In particular we are interested in a problem of the following
nature. Suppose we start with a geometric structure on $M$, which on
its own does not admit a smooth extension to $\barm$, in such a way
that the boundary $\partial M$ is ``at infinity'' in a suitable
sense. Then one may ask whether some weakening of the structure in
question does admit such an extension, and whether this extension
induces a structure on $\partial M$ linked to the interior geometry. A
classical example of this situation, with many applications, involves a
notion of conformal extension. In this case, one starts with a
pseudo--Riemannian metric $g$ on $M$ that does not admit a smooth
extension to $\barm$, for example because it is complete. Then one may
first ask whether the conformal structure $[g]$ on $M$, determined by
$g$, admits a smooth extension to all of $\barm$. Explicitly, this
means that, for each boundary point $x\in\partial M$, there is an open
neighborhood $U\subset\barm$ of $x$ and a smooth nowhere vanishing
function $f:U\cap M\to\Bbb R_{>0}$ such that the pseudo--Riemannian
metric $f g$ on $U\cap M$ admits a smooth extension to all of $U$ for
which the values on $U\cap\partial M$ are non--degenerate as bilinear
forms on tangent spaces.

In most applications, this idea is refined to the more restrictive,
but also more useful, concept of \textit{conformal compactness}.
Rather than assuming an arbitrary smooth rescaling of $g$ extends to
the boundary, one requires that $r^2g$ admits a smooth extension to
the boundary, where $r:U\to\Bbb R_{\geq 0}$ is a \textit{local
  defining function} for the boundary, see section \ref{2.1} for the
formal definition. This enforces a certain uniformity in the growth
rate of the metric $g$ as it approaches the boundary. The property
that $r^2g$ admits a smooth extension to the boundary is independent
of the specific defining function $r$, and it follows that the
conformal class of the induced pseudo--Riemannian metric on $\partial
M$ is also independent of $r$.  This conformal class is then the
induced structure on the boundary, and the boundary so equipped is
then referred to as the \textit{conformal infinity} of the interior.

Alternative to the underlying conformal structure of a
pseudo--Riemannian metric $g$, one may also consider the underlying
projective structure of its Levi--Civita connection $\nabla^g$. The
resulting applications of projective differential geometry to
pseudo--Riemannian geometry have been intensively and successfully
studied during the last years. Because of the resulting emphasis on
geodesic paths, this approach should be particularly useful for
applications in general relativity and scattering. Indeed, as brought
to our attention by P.~Nurowski, there have been attempts to associate
a future time--like projective infinity to space--times, see
\cite{sachs}.

In the setting of a manifold with boundary $\barm=M\cup\partial M$ as
above, one can start from a torsion free linear connection on $TM$,
which does not extend to $\barm$ and assume that its underlying
projective structure extends to $\barm$. Via the Levi--Civita
connection, this concept is then automatically defined for
pseudo--Riemannian metrics. Our first result in this article is an
explicit characterization of extendability of the projective structure
in Proposition \ref{prop-extend}.

In analogy with the way in which conformal compactification usefully
restricts conformal extension (as discussed above) a concept of
\textit{projective compactness}, describing a special type of
projective extension, was introduced in our articles \cite{Proj-comp,CGHjlms}
and studied further in \cite{Proj-comp2}. This involves a parameter
$\al>0$, called the order of projective compactness, and one usually
assumes that $\al\leq 2$. The latter ensures that $\partial M$ is at
infinity according to the parameters of geodesics approaching the
boundary, see Proposition 2.4 in \cite{Proj-comp}. For a projectively
compact connection that preserves a volume density, $\al$
is a measure of growth rate of this volume density toward the boundary at infinity, see \cite{Proj-comp} and Section \ref{2.1}.

There are two results in \cite{Proj-comp} which motivate the
developments in this article. Assume that $\nabla$ is a linear
connection on $TM$, which does not admit a smooth extension to any
neighborhood of a boundary point, but whose projective structure does
admit a smooth extension to all of $\barm$. Then in Theorem 3.3 of
\cite{Proj-comp} it is shown that if $\nabla$ preserves a volume
density and is Ricci flat, then it is projectively compact or order
$\al=1$. On the other hand, if $\nabla$ is the Levi--Civita connection
of a non--Ricci--flat Einstein metric, then Theorem 3.5 of
\cite{Proj-comp} shows that $\nabla$ is projectively compact of order
$\al=2$. In both cases, one actually obtains reductions of projective
holnomy, which lead to much more specific information.

The main result of this article is Theorem \ref{thm:main}, which
provides a vast generalization of the second of these results. Here
being Einstein is replaced by a much weaker condition on the
asymptotics of the scalar curvature of $g$, but we still can conclude
projective compactness of order $\al=2$. Via the results of
\cite{Proj-comp2}, this provides a number of further facts about $g$,
including a certain asymptotic form, an asymptotic version of the
Einstein property, and the fact that $\partial M$ inherits a canonical
conformal structure determined by $g$. Some of these consequences are summarised Corollary \ref{final}.

 The results in Theorem
\ref{thm:main}, along with converse results in \cite{Proj-comp2},
expose a previously unseen critical role for the scalar curvature in
questions of projective compactification.
In Proposition \ref{prop:extend} we show that if the projective
structure of a metric on $M$ extends to $\barm$ then, surprisingly,
its scalar curvature also extends smoothly (as a function) to
$\barm$. The way this works is also important for our treatment.   Several of
the arguments used in proving this result should be of considerable
independent interest.

\section{Results}

\subsection{Projective structures and projective
  compactness}\label{2.1} 

Two torsion free linear connections $\nabla$ and $\hat\nabla$ on the
tangent bundle of a smooth manifold $N$ are called
\textit{projectively equivalent} if they have the same geodesics up to
parametrisation or, equivalently, if there is a one--form
$\Ups\in\Om^1(N)$ such that
\begin{equation}\label{proj-def}
\hat\nabla_\xi\eta=\nabla_\xi\eta+\Ups(\xi)\eta+\Ups(\eta)\xi,
\end{equation}
for vector fields $\xi,\eta\in\frak X(N)$. A \textit{projective
  structure} on $N$  is a projective equivalence class of such
connections. 

Assume that $\barm$ is a smooth manifold with boundary $\partial M$
and interior $M$. Given a torsion free linear connection $\nabla$ on
$TM$, we can define what it means for the projective structure
determined by $\nabla$ to admit a smooth extension to all of $\barm$.
Explicitly, this is the property that, for any boundary point
$x\in\partial M$, there is an open neighborhood $U$ of $x$ in $\barm$
and a one--form $\Ups\in\Om^1(U\cap M)$ such that for all vector
fields $\xi,\eta\in\frak X(U)$ (so these are smooth up to the
boundary), also $\hat\nabla_\xi\eta$ as defined in \eqref{proj-def}
admits a smooth extension to the boundary. It is then clear that the
resulting connection $\hat\nabla$ on $TU$ is uniquely determined up to
projective equivalence, so in this way one indeed obtains an extension
of the projective structure to $\barm$.

More restrictively, for a constant $\al>0$ we say that $\nabla$ is
\textit{projectively compact of order $\al$}, if in the above
considerations the one--form $\Ups$ can be taken to be $\tfrac{dr}{\al
  r}$ for a smooth defining function $r:U\to\Bbb R_{\geq 0}$ for the
boundary. By definition, the latter condition means that
$r^{-1}(\{0\})=U\cap\partial M$ and $dr$ is nowhere vanishing on
$U\cap\partial M$.

If $\nabla$ preserves a volume density, then it has been established in
\cite{Proj-comp} that, in addition to the fact that the projective
structure defined by $\nabla$ extends to $\barm$, projective
compactness of order $\al$ of $\nabla$ only requires a specific growth
rate (related to $\al$) of the parallel volume form towards the
boundary. This can be most conveniently formulated in terms of
defining densities as follows. Observe first, that the notion of a
defining function (for a hypersurface) can be extended to the notion
of a defining section of any real line bundle without problems. The
point here is that for a section of a line bundle, the derivative with
respect to a linear connection is, along the zero--set of the section,
independent of the connection. Hence one can simply require that one
has a section for which the zero--set coincides with the hypersurface
in question and that the derivative of the section with respect to
some linear connection is nowhere vanishing along the zero--set.

Now on any smooth manifold $N$, there is a family of natural line
bundles, obtained from the (trivial) bundle of volume densities by
forming real powers. In the presence of a projective structure, there
is an established convention of \textit{projective weight} for these
line bundles, which are then denoted by $\Cal E(w)$ with $w\in\Bbb R$,
see \cite{BEG}.  The convention is fixed by $\Cal
E(w)=\Cal E(1)^w$ and the fact that if $N$ has dimension $n$, then the
bundle of volume densities on $N$ is $\Cal E(-n-1)$. The following
result is proved in Proposition 2.3 of \cite{Proj-comp}.

\begin{lemma}\label{lem2.1}
Let $\barm$ be a smooth manifold of dimension $n+1$ with boundary
$\partial M$ and interior $M$. Let $\nabla$ be a torsion-free linear
connection on $TM$ whose projective structure admits a smooth
extension to $\barm$ and which preserves a volume density on $M$.

Then $\nabla$ is projectively compact of order $\al$ if and only if
there is a defining density $\si\in\Ga(\Cal E(\al))$ for $\partial M$
such that $\si$ is parallel for $\nabla$ on $M$.  
\end{lemma}

\subsection{Extension of the projective structure}\label{2.2}
Before we move to the main subject of the article, we prove a result
showing that the condition that the projective class of a linear
connection $\nabla$ on $TM$ admits a smooth extension to $\barm$ can
be easily verified explicitly. To the best of our knowledge this
result is not in the literature, although it is related to the
coordinate based strategy for constructing examples in \cite{sachs}
(see expression (2) there).

Given a linear connection $\nabla$ on $TM$ and a chart $U$ with local
coordinates $x^0,\dots,x^n$ for $\barm$, we obtain the connection
coefficients (or Christoffel symbols) $\Ga^k_{ij}\in C^\infty(U\cap
M,\Bbb R)$ for $\nabla$. Denoting by
$\partial_i:=\tfrac{\partial}{\partial x^i}$ the coordinate vector
fields determined by the chart, the connection coefficients are
characterized by
$\nabla_{\partial_i}\partial_j=\sum_k\Ga^k_{ij}\partial_k$, so they
are symmetric in $i$ and $j$. Fixing the local coordinates, the
connections coefficients may be viewed as giving the contorsion tensor
that distinguishes $\nabla$ from the flat connection determined by the
coordinate frame.  
We can thus form the
trace of the connection coefficients $\ga_i:=\sum_k\Ga^k_{ik}$, as
well as form the tracefree part
$\Ps^k_{ij}:=\Ga^k_{ij}-\tfrac1{n+2}(\ga_i\delta^k_j+\ga_j\delta^k_i)$.
Both $\ga_i$ and $ \Ps^k_{ij}$ are smooth, real valued functions on
$U\cap M$ for all $i,j,k=0,\dots,n$.

\begin{prop}\label{prop-extend}
Let $\barm$ be a smooth manifold of dimension $n+1$, with boundary
$\partial M$ and interior $M$, and let $\nabla$ be a linear connection
on $TM$. 

Then the projective class determined by $\nabla$ admits a smooth
extension to $\barm$ if and only if for any point $x\in\partial M$,
there is a local chart $U$ for $\barm$, with $x\in U$, such that the
components of the tracefree part of the connection coefficients of
$\nabla$, with respect to the local coordinates determined by $U$, admit
a smooth extension from $U\cap M$ to all of $U$. 
\end{prop}
\begin{proof}
The fact that a linear connection $\hat\nabla$ on $TM$ admits a smooth
extension to all of $\barm$ is clearly equivalent to the fact that its
connection coefficients in any local chart admit a smooth extension to
the whole domain of the chart. This in turn is equivalent to the same
fact in at least one local chart around each boundary point. Now
suppose that $\nabla$ and $\hat\nabla$ are projectively equivalent as
in \eqref{proj-def} and $\Ups$ is the corresponding one--form. Then
the connection coefficients in a chart $U$ are clearly related by
$$
\hat\Ga^k_{ij}=\Ga^k_{ij}+\Ups_i\delta^k_j+\Ups_j\delta^k_i, 
$$ where $\Ups=\sum_i\Ups_idx^i$. In particular, the tracfree parts of
$\hat\Ga^k_{ij}$ and $\Ga^k_{ij}$ agree. So if the projective class of
$\nabla$ admits a smooth extension to $\barm$, the tracfree parts of
its connection coefficients, with respect to any local chart for
$\barm$ that includes a boundary point, admit a smooth extension to
the boundary.

Conversely, assume that $x\in\partial M$ is a boundary point, and $U$
is a local chart for $\barm$ that contains $x$, and on which the tracefree
parts of the connection coefficients for $\nabla$ admit a smooth
extension to the boundary. Let us denote by $\Ga^i_{jk}$ these
connection coefficients, by $\ga_i$  their trace, and by $\Ps^k_{ij}$ their
tracefree part, as defined above. 
As mentioned above, we may interpret the connection coefficients
$\Ga^i_{jk}$ as the coordinate components of the contorsion tensor
needed to modify the flat connection determined by the chart to the
connection $\nabla$. Hence we can interpret the components
$\Ps^k_{ij}$ of the tracefree part in exactly the same way,
i.e.~define a connection $\hat\nabla$ by
$\hat\nabla_{\partial_i}\partial_j:=\sum_k\Ps^k_{ij}\partial_k$ and by
assumption, this admits a smooth extension to all of $U$. But the fact
that 
$\Ps^k_{ij}:=\Ga^k_{ij}-\tfrac1{n+2}(\ga_i\delta^k_j+\ga_j\delta^k_i)$
on $U\cap M$ shows that the restriction of $\hat\nabla$ to $U\cap M$
is projectively equivalent to $\nabla$.
\end{proof}

\subsection{Extension of the scalar curvature via projective tractors}
\label{2.3}
We now move to the main topic of our article. Consider a
pseudo--Riemannian metric $g$ on $M$ with Levi--Civita connection
$\nabla$. Our standing assumption will be that $\nabla$ does not
extend to any neighborhood of a boundary point. This is for example
implied by completeness of $g$, since this is defined as geodesic
completeness of $\nabla$. Starting from here, we will not use local
coordinates any more and indices showing up will be abstract indices
as introduced by R.~Penrose. In particular, we will denote the metric
$g$ by $g_{ij}$ and its inverse by $g^{ij}$. The Riemann curvature
tensor of $g$ will be denoted by $R_{ij}{}^k{}_\ell$, the
Ricci--curvature of $g$ is $\Ric_{ij}=R_{ki}{}^k{}_j$ 
and its scalar
curvature is $S:=g^{ij}\Ric_{ij}$. We will also use the projective
Schouten tensor $\Rho_{ij}$, which in this simple setting satisfies
$\Rho_{ij}=\tfrac1n\Ric_{ij}$ and its trace
$\Rho=g^{ij}\Rho_{ij}=\tfrac1n S$.

In general, an affine connection is not projectively equivalent to the
Levi--Civita connection of any pseudo--Riemannian metric. The fact
that a projective class contains a Levi--Civita connection is
equivalent to the existence of a non--degenerate solution of a certain
projectively invariant differential equation \cite{Mikes,Sinjukov},
which is sometimes referred to as the metricity equation. The details
of the equation are not important to us, but a discussion in the
notation used here may be found in \cite{CGM,Eastwood-Matveev}.  Let
us briefly discuss some implications of the existence of the solution
to the metricity equation determined by $g$, as well as its
interpretation in terms of tractor bundles.

A projective structure on an $(n+1)$-manifold canonically determines a
invariant linear connection $\nabla^{\mathcal{T}}$. This {\em normal tractor
  connection} is not defined on the tangent bundle but rather on a
rank $(n+2)$-vector bundle known as the {\em standard tractor bundle}
$\Cal T$, see \cite{BEG}.  We write $\Cal T^*$ for its dual, and
$S^2\Cal T$ and $S^2\Cal T^*$ for the symmetric squares of these
bundles. Each of these bundles comes with a composition series in terms
of weighted tensor bundles, which we write as
\begin{equation}\label{comp-ser}
\begin{gathered}
\Cal T=\Cal E(-1)\rpl\Cal E^i(-1) \qquad S^2\Cal T=\Cal E(-2)\rpl\Cal
E^i(-2)\rpl \Cal E^{(ij)}(-2)\\ 
\Cal T^*=\Cal E_i(1)\rpl\Cal E(1) \qquad S^2\Cal T^*=\Cal
E_{(ij)}(2)\rpl\Cal E_i(2)\rpl\Cal E(2) 
\end{gathered}
\end{equation}
Here we use the usual conventions of abstract index notation for
tensor bundles, as well as the convention that adding ``$(w)$'' to the
name of a bundle indicates a tensor product with $\Cal E(w)$. The
composition series for $S^2\Cal T$, for example, means that there are
smooth subbundles $\Cal F_1\subset\Cal F_2\subset S^2\Cal T$ such that
$\Cal F_1\cong\Cal E(-2)$, $\Cal F_2/\Cal F_1\cong \Cal E^i(-2)$ and
$S^2\Cal T/\Cal F_2\cong\Cal E^{(ij)}(-2)$. Choosing a connection in
the projective class, one obtains an isomorphism between each of the
tractor bundles and the direct sum of its composition factors. Having
made such a choice, elements and sections of the bundles will be
denoted as pairs and triples according to the splittings. This
notation is chosen in such a way that, for the composition series
\eqref{comp-ser}, the projecting slot is on top, while the injecting
slot is in the bottom. There are explicit formulae in \cite{BEG} and
\cite{Proj-comp} for the relation between the splittings corresponding
to different connections in the projective class.  See also \cite{EastwoodNotes}
for a further introduction to projective tractor calculus.

The normal tractor connection on $\mathcal{T}$ induces linear
connections on the other tractor bundles, for which we use analogous
notation. Again, explicit formulae for these tractor connections are
available in the references cited above. They are again {\em normal}
meaning that they are associated to the normal conformal Cartan
connection, but below we will usually omit the word ``normal'' except
for where we also consider another connection on the same bundle. 

Now given the metric $g$ on $M$, we denote by $\vol(g)$ its volume
density and we write $\tau:=\vol(g)^{-\frac{2}{n+2}}$, which is a
section of $\Cal E(2)$ defined over $M$ and nowhere vanishing
there. In particular, $\tau^{-1}g^{ij}\in\Ga(\Cal E^{(ij)}(-2))$ is
well defined over $M$, and this is the solution to the metricity
equation determined by $g$. The crucial fact for our purposes is that
there is a corresponding section $L(\tau^{-1}g^{ij})$ of $S^2\Cal T$,
which projects onto $\tau^{-1}g^{ij}$ under the canonical projection,
see \cite{Eastwood-Matveev} or \cite[Proposition 3.1]{CGM}.  The fact
that $\tau^{-1}g^{ij}$ satisfies the metricity equation can be
equivalently characterized as $L(\tau^{-1}g^{ij})$ being parallel for
a natural modification $\nabla^p$ of the tractor connection
$\nabla^{S^2\Cal T}$.  The connection $\nabla^p$ was first constructed in
\cite{Eastwood-Matveev}; a general construction for arbitrary first
BGG--operators is available in \cite{HSSS}. The formula for $\nabla^p$, in the conventions for the splitting we will use below
(which are slightly different from the those in \cite{Eastwood-Matveev}),
is given in Theorem 4.1 of \cite{CGM}. We will not need the detailed
formula, however.

\begin{prop}\label{prop:extend}
Let $g=g_{ij}$ be a pseudo--Riemannian metric on $M$ such that the
projective structure determined by the Levi--Civita connection
$\nabla$ of $g$ admits a smooth extension to all of $\barm$. 

Then the sections $\tau^{-1}g^{ij}\in\Ga(\Cal E^{(ij)}(-2))$, and
$L(\tau^{-1}g^{ij})\in\Ga(S^2\Cal T)$, as well as the scalar curvature
$S$ of $g$ admit smooth extensions to all of $\barm$.
\end{prop}
\begin{proof}
Since we have a well-defined projective structure on $\barm$, all
tractor bundles and tractor connections, as well as the natural
modification $\nabla^p$ of the tractor connection on $S^2\Cal TM$,
are also well-defined on all of $\barm$. By assumption, $L(\tau^{-1}g^{ij})$
is a section of $S^2\Cal T$ defined on the dense open subset
$M\subset\barm$ and parallel for the connection $\nabla^p$
there. Hence it can be extended by parallel transport to a smooth
parallel section over all of $\barm$. The projection of this extension
to the quotient bundle $\Cal E^{(ij)}(-2)$ then provides the claimed
extension of $\tau^{-1}g^{ij}$. 

We can view $L(\tau^{-1}g^{ij})$ as a bundle metric on $\Cal
T^*$. Taking the determinant of of its coefficient--matrix, with
respect to a local frame of $\Cal T^*$, gives rise to a well--defined
section of the bundle $(\La^{n+2}\Cal T^*)^2$. Now this line bundle is
always trivial, with the tractor connection inducing a flat connection
on it. Hence there is a well defined (up to an overall non--zero
constant) determinant of $L(\tau^{-1}g^{ij})$ which is a smooth
function on $\barm$.

Restricting to $M$, we can work in the splittings of $S^2\Cal T$
determined by the Levi--Civita connection $\nabla$. In this splitting,
it is easy to describe $L(\tau^{-1}g^{ij})$ explicitly, since both
$\tau^{-1}$ and $g^{ij}$ are parallel for $\nabla$, see Theorem 3.3 of
\cite{CGM}. One gets
\begin{equation}\label{h-form}
L(\tau^{-1}g^{ij})=
\begin{pmatrix}
  \tau^{-1}g^{ij} \\ 0\\ \tfrac1{n+1}\tau^{-1}g^{ij}\Rho_{ij}
\end{pmatrix}
\end{equation}
 in this splitting. From this one reads off that, over $M$,
$$
\det(L(\tau^{-1}g^{ij}))=\tau^{-n-2}\det(g^{ij})\tfrac1{n+1}g^{ij}\Rho_{ij}.
$$ 
By the definition of $\tau$ we have $\tau^{-n-2}\det(g^{ij})=1$. So,
up to a constant factor, $\det(L(\tau^{-1}g^{ij}))$ equals $S$ on $M$ and  
provides the claimed smooth extension of $S$ to $\barm$.
\end{proof}

\subsection{The non--degenerate case}\label{2.4}
Let $g$ be a pseudo--Riemannian metric on $M$ such that the projective
structure defined by the Levi--Civita connection $\nabla$ of $g$
admits a smooth extension to $\barm$. Then from Proposition
\ref{prop:extend}, we know that the scalar curvature $S$ of $g$ admits
a smooth extension to $\barm$. In what follows, we will primarily be
interested in the case that the resulting boundary value is nowhere
vanishing. This has to happen if, for example,  $S$ is bounded away
from zero on $M$. The proof of Proposition \ref{prop:extend} also
gives a conceptual explanation for the relevance of this condition. We
have seen that $S$ arises as the determinant of the section
$L(\tau^{-1}g^{ij})\in\Ga(S^2\Cal TM)$ associated to the solution
$\tau^{-1}g^{ij}$ of the metricity equation determined by $g$. Hence
our condition exactly means that $L(\tau^{-1}g^{ij})$ is
non--degenerate as a bilinear form on the standard cotractor bundle
along $\partial M$ and hence locally around $\partial M$. 

If this condition is satisfied, we can  consider the pointwise
inverse $\Ph$ of $L(\tau^{-1}g^{ij})$, which is a smooth section of
$S^2\Cal T^*$. In the proof of Proposition \ref{prop:extend}, we have
noted that there is a natural modification $\nabla^p$ of the normal
tractor connection on $S^2\Cal T$ for which $L(\tau^{-1}g^{ij})$ is
parallel. Now the normal tractor connections are all induced by the
standard tractor connection on $\Cal T$, so there is a simple relation
between the derivatives of $\Ph$ and of $L(\tau^{-1}g^{ij})$ with
respect to the tractor connections. We can interpret the fact that
$L(\tau^{-1}g^{ij})$ is parallel for $\nabla^p$ as giving a formula
for the derivative with respect to the tractor connection. From this,
we conclude that $\Ph$ satisfies a differential equation, which will
lead to the main result.

Having chosen a connection $\tilde\nabla$ in the projective class, the
section $\Ph\in\Ga(S^2\Cal T^*)$ has the form
\begin{equation}\label{Phi-form}
\begin{pmatrix} \si \\ \mu_i \\ \ps_{jk}\end{pmatrix} \quad
\text{with\ } \si\in\Ga(\Cal E(2)), \mu_i\in\Ga(\Cal E_i(2)),
\ps_{ij}\in\Ga(\Cal E_{(ij)}(2)). 
\end{equation}
\begin{prop}\label{prop:Phi-equ}
Suppose that the pointwise inverse $\Ph$ of $L(\tau^{-1}g^{ij})$ is
given by \eqref{Phi-form}, in the splitting determined by a connection $\tilde\nabla$
in the projective class. Then we have 
$$
\tilde\nabla_i\si=2\mu_i-2\si\mu_j\xi^j_i-\si^2\eta_i
$$ for certain sections $\xi_i^j\in\Ga(\Cal E_i^j(-2))$ and
$\eta_i\in\Ga(\Cal E_i(-2))$.
\end{prop}
\begin{proof}
During this proof, we will denote all the normal tractor connections
by $\nabla^{\Cal T}$. The only fact about the prolongation connection
$\nabla^p$ we need here is that the modification from $\nabla^{\Cal
  T}$ does not affect the projecting component $\Cal
E^{(ij)}(-2)$, see \cite{Eastwood-Matveev} or \cite{CGM}. 
Otherwise put, writing $\nabla^{\Cal
  T}_iL(\tau^{-1}g^{jk})$ in the given splitting, we get $0$ in the
top slot, and some elements $\xi_i^j\in\Ga(\Cal E_i^j(-2))$ and
$\eta_i\in\Gamma(\Cal E_i(-2))$ in the other two slots. (These can be
described explicitly, see Theorem 4.1 of \cite{CGM}, but we don't need
any details here.)

To describe the relations between the tractor derivatives of
$L(\tau^{-1}g^{ij})$ and $\Ph$, we briefly use abstract index notation
for tractors, writing $h^{AB}$ for $L(\tau^{-1}g^{ij})$ and $\Ph_{AB}$
for $\Ph$. Then by definition
$\Ph_{AC}h^{CB}=\delta^A_B$. Differentiating this, we get
$h^{CB}\nabla^{\Cal T}_i\Ph_{AC}=-\Ph_{AC}\nabla^{\Cal T}_ih^{CB}$,
which immediately implies that
$$
\nabla^{\Cal  T}_i\Ph_{AB}=-\Ph_{AC}\Ph_{DB}\nabla^{\Cal T}_ih^{CD}. 
$$ Now the projecting slot of $\nabla^{\Cal T}_i\Ph$ in the given
splitting is $\tilde\nabla_i\si-2\mu_i$, see Section 3.1 of
\cite{Proj-comp}. To compute the projecting slot of the right hand
side, we can take two sections of $\Cal E(-1)\subset\Cal T$, convert
them to sections of $\Cal T^*$ using $\Ph$, and then hook them into the
bilinear form defined by $\nabla_iL(\tau^{-1}g^{ij})$. One easily
verifies that the bundle map $\Cal T\to\Cal T^*$ induced by $\Ph$
sends an element of the form $\binom0\alpha$ to
$\binom{\al\si}{\al\mu_i}$. Hooking this into the bilinear form, one
immediately sees that the projecting slot of the right hand side
equals $2\xi_i^j\mu_j\si+\eta_i\si^2$, which implies the
claim. 
\end{proof}

\subsection{The main result}\label{2.5}
Now we are ready to prove our main result. 

\begin{thm}\label{thm:main}
Let $\barm$ be a smooth manifold of dimension $n+1$ with boundary
$\partial M$ and interior $M$. Let $g=g_{ij}$ be a pseudo--Riemannian
metric on $M$ with Levi--Civita connection $\nabla$, which has the
following properties.
\begin{itemize}
\item The projective structure on $M$ defined by $\nabla$ admits a
  smooth extension to $\barm$. 
\item The connection $\nabla$ itself does not admit a smooth extension
  to any neighborhood of a boundary point. 
\item The scalar curvature $S$ of $g$ is bounded away from zero, or,
  more generally, the boundary value of the smooth extension of $S$ to
  $\barm$, as guaranteed by Proposition \ref{prop:extend}, is nowhere
  vanishing.
\end{itemize}
Then $g$ is projectively compact of order $\al=2$. 
\end{thm}
\begin{proof}
  As in Section \ref{2.3}, we put $\tau:=\vol(g)^{-\frac{2}{n+2}}$,
  which is a section of $\Cal E(2)$ defined over $M$ and nowhere
  vanishing there. Denoting by $g^{ij}$ the inverse of $g_{ij}$, we
  know from Section \ref{2.3} and Proposition \ref{prop:extend} that
  the section $\tau^{-1}g^{ij}$ of $\Cal E^{ij}(-2)$, which is
  initially defined over $M$, is a solution of the metricity equation
  and admits a smooth extension to $\barm$. Then the corresponding
  section $L(\tau^{-1}g^{ij})\in\Ga(S^2\Cal T)$ is parallel for the
  connection $\nabla^p$.

  In Section \ref{2.4}, we have noted that our assumption on $S$
  implies that, as a bilinear form on $\Cal T^*$, $L(\tau^{-1}g^{ij})$
  is non--degenerate along $\partial M$, and hence locally around
  $\partial M$. In the further considerations, we can restrict to a
  neighborhood of $\partial M$ where this is true, i.e.~we will assume
  that $L(\tau^{-1}g^{ij})$ is non--degenerate on all of $\barm$. Then
  we can form the inverse bundle metric $\Ph\in\Ga(S^2\Cal T^*)$ as in
  \ref{2.4}.

Over $M$, we can work in the splitting associated to $\nabla$, and
there we have the expression \eqref{h-form} for
$L(\tau^{-1}g^{ij})$. This immediately implies that, in the splitting
determined by $\nabla$, we have
\begin{equation}
  \label{Phiform}
  \Ph=\begin{pmatrix} (n+1)\tau(g^{ij}\Rho_{ij})^{-1} \\ 0 \\ \tau
    g_{ij}\end{pmatrix}, 
\end{equation}
so the top slot of this is a non--zero multiple of $\tau
S^{-1}$. Since this is the projecting slot, it is actually independent
of the choice of splitting, and passing to the splitting associated to
a connection $\tilde\nabla$ in the projective class which is smooth up
to the boundary, we conclude that $\tau S^{-1}$ admits a smooth
extension to all of $\barm$, so by Proposition \ref{prop:extend},
$\tau$ admits a smooth extension to $\barm$. 

Next, we claim that this smooth extension vanishes along the boundary
$\partial M$. Suppose that $x\in\partial M$ is a point such that
$\tau(x)\neq 0$. Then choose an open neighborhood $U$ of $x$ in
$\barm$ on which $\tau$ is nowhere vanishing. It is well known that
there is a unique connection $\hat\nabla$ in the restriction of the
projective class to $U$, such that $\tau|_U$ is parallel for the
induced connection on $\Cal E(2)$. But then over $U\cap M$, both
$\nabla$ and $\hat\nabla$ preserve $\tau$ and hence have to
agree. This contradicts the assumption that $\nabla$ does not extend
smoothly to any neighborhood of a boundary point.

Now we finally claim that $\tau\in\Ga(\Cal E(2))$ is a defining
density for $\partial M$, which in view of Lemma \ref{lem2.1}
completes the proof. Taking a connection $\tilde\nabla$ in the
projective class which is smooth up to the boundary as above, we have
to prove that $\tilde\nabla\tau$ is nowhere vanishing along $\partial
M$. From above we know that the top slot of $\Ph$ with respect to any
splitting is given by a non--zero constant multiple of $\tau
S^{-1}$. In particular, this top slot vanishes along $\partial M$. By
Proposition \ref{prop:Phi-equ}, we conclude that, along $\partial M$,
the middle slot of $\Ph$ in this splitting has to be a non--zero
multiple of $\tilde\nabla_i(\tau
S^{-1})=S^{-1}\tilde\nabla_i\tau+\tau\tilde\nabla_iS^{-1}$. Of course,
the second summand vanishes along the boundary, so there the middle
slot equals $\tilde\nabla_i\tau$, up to multiplication by a
nowhere--vanishing function. But we know that $\Ph$ is the inverse of
$L(\tau^{-1}g^{ij})$, so in particular it is non--degenerate as a
bilinear form (on $\mathcal{T}$) over all of $\barm$. By
non--degeneracy, vanishing of the top slot along $\partial M$ implies
that the middle slot has to be nowhere vanishing along $\partial M$.
\end{proof}

Knowing that a metric $g$ satisfying the assumptions of Theorem
\ref{thm:main} is projectively compact of order $\al=2$ allows
one to apply all the results of \cite{Proj-comp2} to $g$. We summarise
here some of the key points from there (that are mainly drawn from Theorems 7 and 11 in that source):

\begin{cor}\label{final}
  Let $\barm$ be a smooth $(n+1)$-manifold with boundary $\partial M$
  and interior $M$. Let $g=g_{ij}$ be a pseudo--Riemannian metric on
  $M$ which satisfies the three hypotheses of Theorem
  \ref{thm:main}. Then:

  (1) The smooth extension $S$ of the scalar curvature of $g$ to all
  of $\barm$ has a boundary value which is locally constant and
  nowhere vanishing.

  (2) Given any boundary point $x\in\partial M$ and local defining
  function for $\partial M$, there is a neighbourhood of $x$ on which
  $g$ admits the asymptotic form
$$
g=\frac{Cd\rho^2}{\rho^2}+\frac{h}{\rho},
$$
where $C=\tfrac{-n(n+1)}{4S(x)}$ (and so is constant) and $h$
is a symmetric $\binom02$--tensor field which admits a smooth
extension to the boundary, with boundary values being non--degenerate
on $T\partial M$.

(3) There is a canonical conformal structure on $\partial M$. This is
given locally by the conformal class of the restriction of $h$ to
$T\partial M$.

(4) The metric $g$ is asymptotically Einstein in the sense that the
trace--free part $R_{ab}-\tfrac{S}{n+1}g_{ab}$ of the Ricci tensor
admits a smooth extension to the boundary (while the Ricci tensor
evidently blows up).
\end{cor}

\end{document}